\documentclass[12pt,reqno]{amsart}
\usepackage{amssymb}

\usepackage[curve,matrix,arrow]{xy}

\baselineskip

\theoremstyle{plain} 
\numberwithin{equation}{section}
\newtheorem{thm}[equation]{Theorem}
\newtheorem{lemma}[equation]{Lemma}

\newtheorem{prop}[equation]{Proposition}

\theoremstyle{definition}
\newtheorem{rem}[equation]{Remark}

\theoremstyle{remark}

\def\CC{{\mathcal{C}}}

\def\CE{{\mathcal{E}}}
\def\CF{{\mathcal{F}}}

\def\CM{{\mathcal{M}}}

\def\CS{{\mathcal{S}}}

\def\CW{{\mathcal{W}}}

\def\CV{{\mathcal{V}}}

\def\CV{{\mathcal{V}}}

\def\bP{{\mathbb P}}

\def\Proj{\operatorname{Proj}\nolimits}

\def\res{\operatorname{res}\nolimits}

\def\Rad{\operatorname{Rad}\nolimits}

\def\modg{\operatorname{{\bf mod}(\text{$kG$})}\nolimits}

\def\Modg{\operatorname{{\bf Mod}(\text{$kG$})}\nolimits}

\def\stmodg{\operatorname{{\bf stmod}(\text{$kG$})}\nolimits}
\def\Stmodg{\operatorname{{\bf StMod}(\text{$kG$})}\nolimits}

\def\stmodh{\operatorname{{\bf stmod}(\text{$kH$})}\nolimits}
\def\stmod{\operatorname{{\bf stmod}}\nolimits}

\def\HHH{\operatorname{H}\nolimits}
\def\Hom{\operatorname{Hom}\nolimits}
\def\PHom{\operatorname{PHom}\nolimits}

\def\End{\operatorname{End}\nolimits}
\def\HH#1#2#3{\HHH^{#1}(#2,#3)}

\def\Homul{\operatorname{\underline{Hom}}\nolimits}

\def\Ext{\operatorname{Ext}\nolimits}

\def\hgs{\HH{*}{G}{k}}
\def\hes{\HH{*}{E}{k}}
\def\hhs{\HH{*}{H}{k}}

\def\VG{{V_G(k)}}

\def\VE{{V_E(k)}}

\def\sfk{{\mathsf k}}

\title[Endomorphism ring of the trivial module]
{Negative cohomology and the endomorphism ring of the trivial module}
\author[Jon F. Carlson]{Jon F. Carlson}
\address{Department of Mathematics, University of Georgia, 
Athens, GA 30602, USA}
\email{jfc@math.uga.edu}
\thanks{Research partially supported by 
Simons Foundation grant 054813-01}

\date\today
\subjclass{20C20 (primary), 20J06, 18G80}

\begin{document}

\begin{abstract} 
Let $k$ be a field of characteristic $2$ and let $H$ be a finite group
or group scheme. 
We show that the negative Tate cohomology ring $\widehat{\HHH}^{\leq 0}(H,k)$
can be realized as the endomorphism ring of the trivial module in a 
Verdier localization of the stable category of $kG$-modules for 
$G$ an extension of $H$. This means in some cases that the endomorphism of
the trivial module is a local ring with infinitely generated radical 
with square zero. This stands in stark contrast to some known calculations
in which the endomorphism ring of the trivial module is the degree zero 
component of a localization of the cohomology ring of the group. 
\end{abstract}

\maketitle

\section{Introduction}
Let $G$ be a finite group and $k$ a field of characteristic $p >0$. 
Let $\stmodg$ be the stable category of finitely generated $kG$-modules 
modulo projective modules. It is a tensor triangulated category and its
thick tensor ideal subcategories have been classified in terms
of support varieties. Given a thick tensor ideal $\CM$,
a new category $\CC$ is obtained by localizing $\stmodg$ at $\CM$ 
by inverting any map whenever the third object in the triangle of that 
map is in $\CM$. In such a category, the endomorphism of the trivial module
$\End_{\CC}(k)$ is important because the 
endomorphism ring of every module is an algebra
over it. In favorable cases its action on other modules can be used to 
define support varieties and other invariants.

In the few cases where $\End_{\CC}(k)$ has
been computed, it turned out to be some sort of homogeneous localization 
of the cohomology ring of the group (see \cite{R}, \cite{CW}, \cite{BG}). 
However, in all of those examples, the 
localization is with respect to a thick tensor ideal determined by 
a subvariety that is a hypersurface or union of hypersurfaces. It was 
clear that the technique in those examples used for
determining $\End_{\CC}(k)$ does not work even when the determining
subvariety is a point in $\bP^2$. The point of this paper is to show that
in such a case the structure of $\End_{\CC}(k)$ can be very different.
In examples, we find that  $\End_{\CC}(k)$ is a local $k$-algebra 
whose maximal ideal is infinitely generated and has square zero. 

We work in the setting that $G = H \times C$ is a finite group 
scheme where $H$ and $C$ are subgroup schemes defined over $k$ 
and the group algebra $kC$ is the group algebra of 
a cyclic group of order $p$. The thick tensor ideal subcategory is collection
of all finitely generated $kG$-modules whose variety is the 
image of the variety of $C$ under the restriction map. Thus a module in $\CM$
is projective on restriction to a $kH$-module, but not projective when 
restricted to $kC$. In this setting we prove that
\[
\End_{\CC}(k)  \cong \widehat{\HHH}^{\leq 0}(H,k),
\]
the negative Tate cohomology ring of the subgroup scheme $H$. 
In the course of the proof we construct
the idempotent modules associated to the subcategory $\CM$. These are 
constructed directly from a $kH$-projective resolution of the 
trivial module for $H$, and it is this resolution that connects us to 
the Tate cohomology. In the case of group algebras, it has been shown
that products in negative cohomology mostly vanish \cite{BC2} whenever
the $p$-rank of the group is at least $2$. It is likely that the same 
holds for general group schemes whenever the Krull dimension of the 
cohomology ring of $H$ is at least $2$.  


\section{Preliminaries} \label{sec:prelim}
In this section we establish some notation and recall some known results.
For general reference on cohomology see \cite{CTVZ} or \cite{Cmods}. For
basics on triangulated categories see \cite{Hap}. We follow that 
development in \cite{FP} for support varieties. 

Let $k$ be a field of characteristic $p > 0$, and let $G$ be a finite group
scheme defined over $k$. Let $kG$ be its group algebra.  
Let $\modg$ denote the category of finitely generated $kG$-modules and
$\Modg$ the category of all $kG$-modules.
Recall that $kG$ is a cocommutative Hopf algebra which means that for
$M$ and $N$, $kG$-modules, $M \otimes_k N$ is also a $kG$-module with action
given by coalgebra map $kG \to kG \otimes kG$. If $G$ is a finite group, 
then $g(m\otimes n) = gm \otimes gn$ for $g \in G$, $m \in M$ and $n \in N$.
By the symbol $\otimes$ we mean $\otimes_k$ unless otherwise indicated.
In addition, $kG$ is self-injective so that projective $kG$-modules coincide
with injective $kG$-modules.

For $M$ a $kG$-module, $\Omega^{-1}(M)$ is the cokernel of an injective
hull $M \hookrightarrow I$, for $I$ 
injective. Inductively, we let $\Omega^{-n}(M) =
\Omega^{-1}(\Omega^{1-n}(M)$, On the positive side, $\Omega(M)$ is
the kernel of a projective cover $P \twoheadrightarrow M$, for $P$ projective,
and $\Omega^n(M) = \Omega(\Omega^{n-1}(M).$

The stable category $\stmodg$
of $KG$-modules modulo projectives has the same objects as $\modg$, but
the morphisms are given by the formula 
\[
\Homul_{kG}(M,N) \ = \ \Hom_{\stmodg}(M,N) \ = \ \Hom_{kG}(M,N)/\PHom_{kG}(M,N)
\]
where $\PHom_{kG}$ is the set of homomorphisms that factor through a
projective module. The stable category is a tensor triangulated
category. Triangles correspond to 
short exact sequences in $\modg$. The translation functor is $\Omega^{-1}$.
Let $\Stmodg$ denote the stable category of all $kG$-modules. It has
the same properties. 

The cohomology ring $\hgs$ is a finitely generated $k$-algebra, and every
cohomology module $\Ext^*_{kG}(M,N)$ is a finitely generated module
over $\hgs$ for $M$ and $N$ in $\modg$ \cite{FS}. 
Let $\VG= \Proj \hgs$ denote the projectivized
prime ideal spectrum of $\hgs$. If $E$ is an elementary abelian $p$-group
or rank $r$ (order $p^r$), 
then modulo its radical $\hes/\Rad(\hes) \cong k[\zeta_1, \dots, \zeta_r]$
is a polynomial ring in $r$ variables. Thus when the field is algebraically
closed, $\VE \cong \bP^{r-1}$ is projective $r-1$ space.

We define the support variety of a $kG$-module by the method of $\pi$-points
\cite{FP}. For finite groups this is essentially the same as the development
in \cite{BCR2}. A $\pi$-point for $G$ is a flat map $\alpha_K: K[t]/(t^p) \to
KG_K$, where $K$ is an extension of $k$, and $\alpha_K$ factors through
the group algebra of some unipotent abelian subgroup
scheme $C_K \subseteq G_K$
of $G_K$. For $M$ a $kG$-module, let $\alpha_K^*(M_K)$ denote the restriction
of $M_K = K \otimes M$ to a $K[t]/(t^p)$-module along $\alpha_K$. Two 
$\pi$-points $\alpha_K$ and $\beta_L$ are equivalent if for every finite
dimensional $kG$-module $M$, $\alpha_K^*(M_K)$ is projective if and only if 
$\beta_L^*(M_L)$ is projective. 

We say that a $\pi$-point $\alpha_K$ specializes to $\beta_L$ if, for any 
finitely generated $kG$-module $M$, the projectivity of $\alpha^*_K(M)$
implies the projectivity of $\beta^*_L(M)$.  So two $\pi$-points are 
equivalent if each specializes to the other. 

Let $\CV_G(k)$ denote the set of all equivalence classes of $\pi$-points. 
Then $\CV_G(k)$ is a scheme and is isomorphic as a scheme to $\VG$. 
Essentially, the class of a $\pi$-point $\alpha_K$ corresponds to the 
homogeneous prime ideal that is the kernel of the restriction map 
$\HHH^*(G,K) \to \HHH^*(K[t]/(t^p), K)/ \Rad(\HHH(K[t]/(t^p))$ along
$\alpha_K$. Thus, since we extend the field $k$, a $\pi$-point may 
correspond to generic point of a homogeneous irreducible subvariety
of $\hgs$. The 
support variety $\CV_G(M)$ of a $kG$-module $M$ is the set of 
all equivalence classes of $\pi$-points $\alpha_K$  
such that $\alpha_K^*(M_K)$ is not projective. 
If $M$ is a finitely generated module then $\CV_G(M)$ is a closed subvariety 
of $\CV_G(k)$. Otherwise, it is just a subset. 

A subcategory of a tensor triangulated category is thick if it is closed under
taking of direct summands. It is a thick tensor ideal if,
in addition, the tensor product
of an object in the subcategory with any other object is again in the
subcategory. There is a complete classification of the
thick tensor ideals of $\stmodg$. Suppose that $\CV$ is a subset of
$\CV_G(k)$ that is closed under specializations, 
meaning that if $\alpha_K$ specializes to $\beta_L$ and if 
$\alpha_K$ is in $\CV$ then so is $\beta_L$. Let $\CM_{\CV}$ be
the full subcategory of $\stmodg$ generated by all $kG$-modules $M$
such that  $\CV_G(M) \subseteq \CV$.
The properties of the support variety are sufficient to insure that
any $\CM_{\CV}$ is a thick tensor ideal in $\stmodg$.

\begin{thm}  \cite{BCR, FP} \label{thm:bcr3}
Every thick tensor ideal in $\stmodg$ is equal to $\CM_{\CV}$ for some
some subset $\CV \subseteq \CV_G(k)$ which is closed under specialization.
\end{thm}

If $\CM$ is a thick subcategory of a triangulated category $\CC$, then
the Verdier localization of $\CC$ at $\CM$ is the category whose objects
are the same as those of $\CC$ and whose morphisms are obtained by inverting
a morphism if the third object in the triangle of that morphism is in
the subcategory $\CM$. Thus, a morphism from $L$ to $N$
in the localized category has the form 
\[
\xymatrix{
N  \ar[r]^\theta & M  & L \ar[l]_\gamma
}
\]
where the third object in the triangle of the map $\theta$ is in $\CM$.
So in the localized category $\theta^{-1}\gamma$ is a
morphism. 


\section{Resolution modules} \label{sec:resol}
Assume that  $k$ is a field of characteristic $p$, and 
consider a finite group scheme of the form $G = H \times C$ 
where  $C$  is the group scheme of a cyclic  group of order $p$
and $H$ is finite group scheme defined over $k$.
Let $z$ be a generator for $C$, and $Z = z-1$, so that 
$kC \cong k[Z]/(Z^p)$ and $kG \cong kH \otimes k[Z]/(Z^p)$. 

Suppose that 
\[
\xymatrix{
\dots \ar[r] &C_2 \ar[r]^{\partial} & C_1 \ar[r]^{\partial} & 
C_0 \ar[r] &  0
}
\]
is a complex of $kH$-module, with all $C_i = \{0\}$ for $i <0$. 
We use the complex to define a sequence of $kG$-modules 
which we call resolution
modules. For any $n > 0$, let $M(P_*,n)$
be the $kG$-module whose restriction to $kH$ is the infinite direct sum
\[
C_0 \oplus C_1^{p-1} \oplus C_2 \oplus C_3^{p-1} \oplus \dots 
\oplus C_{2n-1}^{p-1}.
\]
For $i$ odd, define the action of $Z$ on 
$(m_1, \dots, m_{p-1}) \in C_i^{p-1}$ to be
\[
Z(m_1, \dots, m_{p-1}) = (0, m_1, \dots, m_{p-2}) + \partial(m_{p-1}) 
\quad \in C_i^{p-1} \oplus C_{i-1}
\] 
while for $i = 2j$ and $m \in C_i$, let
\[
Zm = \begin{cases}
0 & \text{ if } m \in C_0 \text{  and} \\
(\partial(m), 0, \dots, 0) \in C_{2j-1}^{p-1} & \text{ if } 
m \in C_{2j},  \  j >0 . 
\end{cases}
\]
Note that the action of $Z$ commutes with that of $kH$, since
the boundary maps $\partial$ are $kH$-homomorphisms. Moreover,
$Z^pM = \{ 0\}$ because $C_*$ a complex. Consequently, the relations
define a $kG$-module.

We have a nested sequence of modules 
\[
M(C_*,1) \subseteq M(C_*,2) \subseteq M(C_*,3)
\subseteq \dots \subseteq  M(C_*,\infty).
\]
where $M(C_*,\infty)$ is the limit. 
That is $M(C_*,\infty)$ is the module whose restriction to $H$ is 
$\oplus_{i \geq 0} (C_{2i} \oplus C_{2i+1}^{p-1})$ 
with the action by $Z$ defined as above. 

Note that if $C_*$ is a complex of finitely generated modules, then 
every $M(C_*,n)$ is finitely generated,  though $M(C_*,\infty)$ may not be. 

The construction has several interesting properties. 

\begin{lemma} \label{lem:resol1}
Suppose that $C_*$ and $D_*$ are chain complexes 
of $kH$-modules in nonnegative degrees.
We have the following. 
\begin{enumerate}
\item  Any chain map $\sigma: C_* \to D_*$ in even degrees
induces a homomorphism $M(\sigma):
M(C_*,n) \to M(D_*,n)$ for every $n\geq 0$ and $n = \infty$.
\item For every $n\geq 0$ and for $n = \infty$, $M(C_* \oplus D_*,n)
\cong M(C_*,n) \oplus M(D_*,n)$.
\item If $J$ is a subgroup scheme of $H$, then the restriction of 
$M(C_*, n)$ to $J \times C$ is isomorphic to $M((C_*)_{\downarrow kJ}, n)$
where $(C_*)_{\downarrow kJ}$ is the restriction of $C_*$ to a complex of
$kJ$-modules.
\item If $C_*$ is an exact complex of projective modules, then in the 
stable module category $M(C_*,\infty)$ is zero. 
\item If $C_*$ and $D_*$ are projective resolutions of the same module $N$,
then in the stable category 
$M(C_*,\infty) \cong M(D_*,\infty)$.
\end{enumerate}
\end{lemma}

\begin{proof}
The first item is clear since any chain map commutes with the boundary map
and hence the $kH$-map $M(\sigma):M(C_*,n) \to M(D_*,n)$, which is 
defined on the 
direct sum of the terms of the complex, is a $kG$-homomorphism. The proofs of 
Items 2 and 3 are straightforward. 

Suppose that $C_*$ is an exact complex of projective modules. Then $C_*$ is a
direct sum of complexes having the form  
$0 \to D_{i+1} \to D_i \to 0$. For
$n > i$, either $M(D_*,n) \cong D_{i+1}^{p-1} \oplus D_i$ 
or $M(D_*,n) \cong D_{i+1} \oplus D_i^{p-1}$ 
as $kH$-modules. Because  $\partial$ maps
$D_{i+1}$ isomorphically onto $D_i$ and $D_i$ is 
projective as a $kH$-module,  we conclude that 
$M(D_*,n) \cong D_i \otimes k[Z]/(Z^p)$ is a 
projective module. 

Suppose that $C_*$ and $D_*$ are projective resolutions of the same module $N$.
Then there are chain maps $\sigma: C_* \to D_*$ and $\tau: D_* \to C_*$ 
that lift the identity of $N$. That is, the compositions $\sigma\tau$
and $\tau\sigma$ are homotopic to the identity maps on $D_*$ and $C_*$, 
respectively. It follows that there are exact complexes $P_*$ and $Q_*$
of projective $kH$-modules such that $C_* \oplus P_* \cong D_* \oplus Q_*$.
Thus part (4) follows from parts (2) and (3). 
\end{proof}

From the above we see that there is a functor $\Gamma: \stmod(kH) 
\to \Stmodg$ that takes a $kH$-module $N$ to $M(P_*,\infty)$ where 
$P_*$ is a $kH$-projective resolution of $N$. Moreover, it can be checked
that this is a functor of triangulated categories, since for any map 
between objects in $\stmodh$, the mapping cone of the induce map on 
projective resolutions is a projective resolution of the third object 
in the triangle of that map. 

What is interesting, is that if we assume that $p=2$ and adjust 
the Hopf algebra structure on $kG$, then $\Gamma$ is also a functor
of tensor triangulated categories. That is, the normal coalgebra 
structure on the group algebra 
$kC$ is the diagonal which takes a group element $g$ 
to $g \otimes g$. Because $kG$ is a product of algebras $kG \cong kH \otimes 
k[Z]/(Z^2)$, there is another natural coalgebra map that 
is the given coalgebra map on $H$ and takes $Z$ to $1 \otimes Z +
Z \otimes 1$. This comes by regarding $k[Z]/(Z^2)$ as the restricted
enveloping algebra of a one dimensional restricted Lie algebra. 
Regarding it as a group algebra of a cyclic group of order 2, we 
would have that $Z \mapsto 1 \otimes Z + Z \otimes 1 + Z \otimes Z$. 

Now note that if $P_*$ and $Q_*$ are projective resolutions of 
$kH$-modules $L$ and $N$, then in $M(P_* \otimes Q_*, \infty)$ 
we have that $Z(p\otimes q) = p \otimes Zq + Zp \otimes q =
p \otimes \partial(q) + \partial(p) \otimes q = 
\partial(p \otimes q)$. Thus we have, using the Lie coalgebra
structure,  that 
\[
\Gamma(L \otimes M) = M(P_* \otimes Q_*, \infty) \cong
M(P_*, \infty) \otimes M_(Q_*, \infty) = \Gamma(L) \otimes \Gamma(N).
\]

In addition, it implies that, with the Lie coalgebra structure,
\[
\Gamma(k) \otimes \Gamma(k) \cong \Gamma(k \otimes k) = \Gamma(k)
\]
so that $\Gamma(k)$ is an idempotent module. In the Section \ref{sec:cann}, 
we show that $\Gamma(k)$ is idempotent even without the change in the Hopf 
structure. 


\section{An exact sequence} \label{sec:exact} 
Assume that $k$ and $G = H \times C$ are as before. 
The purpose of this section is to construct a triangle in $\Stmodg$ that 
in Section \ref{sec:cann} is shown to be the canonical triangle of 
idempotent modules associated to the thick subcategory of $\stmodg$
consisting of modules whose support varieties are in the image of 
$\CV_C(k)$ in $\CV_G(k)$. Let $\sfk$ denote the complex of 
$kH$-modules that has $k$ in degree $0$ and the zero module in all other 
degrees. 

Suppose that $P_*$
is a projective resolution of the trivial $kH$-module $k$. Let $C_*$ 
be the augmented complex in nonnegative degrees:
\[
\xymatrix{ 
\dots \ar[r] & P_2 \ar[r]^{-\partial} & P_1 \ar[r]^{-\partial}
& P_0 \ar[r]^{-\varepsilon} & k \ar[r] & 0
}
\]
with the augmentation and boundary maps negated. That is, we set $C_0 = k$
and $C_i = P_{i-1}$, for $i > 0$.   Let $N(P_*, n)$ be the module whose 
restriction to a $kH$-module is the direct sum
\[ 
k \oplus P_0^{p-1} \oplus P_1 \oplus P_2^{p-1} \oplus \dots \oplus P_{2n-1}
\]
Multiplication by $Z$ annihilates the direct summand $k$. For
$m = (m_1, \dots, m_{p-1}) \in P_{2i}^{p-1}$, define
\[
Zm = \begin{cases} - \varepsilon(m_{p-1}) + (0, m_1, \dots, m_{p-2}) 
\in k \oplus P_0^{p-1} & \text{ if } i = 0 \\
-\partial(m_{p-1}) + (0, m_1, \dots, m_{p-2}) \in P_{2i-1} \oplus P_{2i} & 
\text{ if } i > 0  \end{cases}
\]
For $m \in P_{2i-1}$, let 
$Zm = -(\partial(m), 0, \dots, 0) \in P_{2i-2}^{p-1}$. 
Thus, $N(P_*,n)$ looks like $N(C_*,n)$ except that it has 
an odd rather than even number of $kH$-summand.  
In the limit, $N(P_*,\infty) \cong M(C_*, \infty)$.

It is easy to see that Lemma \ref{lem:resol1} 
holds for this construction. 
The main result of this section is the construction of an exact sequence with 
the form given in the next proposition.

\begin{prop} \label{prop:canon-tri}
For the projective resolution $P_*$ as above and any $n>0$, including 
$n = \infty$, there is a
projective $kG$-module $Q$ and an exact sequence 
\[
\xymatrix{
0 \ar[r] & M(P_*,n) \ar[r]^{\theta} & k \oplus Q \ar[r]^{\mu}
& N(P_*,n) \ar[r] & 0
}
\]
where the class of the map $\theta$ in $\Homul_{kG}(M(P_*,n),k)$ is 
the class of the map induced by the augmentation $\varepsilon:P_* \to k$
and the class of the map $\mu$ in $\Homul_{kG}(k, N(P_*,n))$ is the class 
of the map induced by the degree zero inclusion of $\sfk$ into 
the augmented projective resolution $(P_*, \varepsilon)$. 
\end{prop}

\begin{proof}
For $i \geq 0$, let $Q_i = P_i \otimes k[Z]/(Z^p)$ 
which is a projective module over
$kG = kH \otimes k[Z]/(Z^p)$. Let $Q = Q_0 \oplus \dots \oplus Q_{2n-1}$ or 
let $Q = Q_0 \oplus Q_1 \oplus \dots$ in the case that $n = \infty$. 
We define the maps $\theta$ and $\mu$ as follows. 
Let $\ell_Q$ and $\ell_N$ be a generator for the $kH$-summand isomorphic
to $k$ in $k \oplus Q$ and in $N(P_*,n)$, respectively. Then 
\[
\theta(m) = \varepsilon(m)\ell_Q \otimes 1 \quad \oplus \quad 
m \otimes Z^{p-1} \qquad \in 
k\oplus Q_{0} \quad \text{ for } m \in P_{0}.
\]
For $0 < i < n$, let
\[
\theta(m) = \partial(m) \otimes 1 \quad \oplus \quad m \otimes Z^{p-1} 
\qquad \in Q_{2i-1} \oplus Q_{2i} \quad \text{ for } m \in P_{2i}.
\]
For $1 \leq i \leq n$ and $(m_1, \dots, m_{p-1}) \in P_{2i}^{p-1}$, let
\[
\theta(m_1, \dots, m_{p-1}) = \sum_{j = 1}^{p-1} \partial(m_j) \otimes Z^{j-1} 
\ \oplus \ \sum_{j = 1}^{p-1} m_j \otimes Z^{i}
\quad \in Q_{2i-1} \oplus Q_{2i}.
\]
The map $\mu$ is given by the following rules. First, let $\mu(\ell_Q) 
= \ell_N$. For an element $\sum_{j=0}^{p-1}  m_j \otimes Z^j \in Q_i$, let
\[
\mu(\sum_{j=0}^{p-1} m_j \otimes Z^j) = \begin{cases}
-\varepsilon(m_{p-1}) \oplus (m_0, \dots, m_{p-2}) & \in k \oplus P_0^{p-1}
\quad \text{ if } i=0 \\
-\partial(m_{p-1}) \oplus (m_0, \dots, m_{p-2}) & \in P_{i-1} \oplus P_i^{p-1}
\quad \text{ if } i \text{ is even} \\
-(\partial m_, \dots, \partial m_{p-1}) \oplus m_0 &
\in P_{i-1}^{p-1} \oplus P_i  \quad \text{ if } i \text{ is odd}
\end{cases} 
\]
It is easy to see that $\theta$ and $\mu$ are $kH$-homomorphisms. Hence to 
see that they are $kG$-homomorphisms, it is only necessary to show that the
maps commute with the action of $Z$. 
We leave it also to the reader to check that 
$\mu\theta = 0$. Once this is done, 
the exactness of the sequence can be demonstrated by noting that there is 
an obvious filtration on the sequence itself such that the successive 
quotients have the form either 
$0 \to P_i \to k \oplus Q_i \to P_i^{p-1} \to 0$
or $0 \to P_i^{p-1} \to k \oplus Q_i \to P_i \to 0$ where the maps are 
induced by $\theta$ and $\mu$. One can show that these quotient 
sequences are exact. Finally, the maps to and from the summand $k$ in 
the middle term of the sequence can be determined to be as asserted from
the construction. 
\end{proof}


\section{Idempotent Modules} \label{sec:idmod}
In this section, we review some information that we require on idempotent
modules. We also give a brief description of a calculation of the 
endomorphism ring of the trivial module in the stable category localized
at a thick subcategory defined by the subvariety of  an ideal 
generated by a single element in $\hgs$ (see Theorem \ref{thm:bcr3}). 
This material is taken mostly from Rickard's paper \cite{R}.
Variations on the theme and other accounts, can be found in \cite{CW}, 
\cite{Cmods} and the last section of \cite{BG})

In general, associated to a thick tensor ideal $\CM$ in $\stmodg$, for 
any $X$ in $\Stmodg$, there is a 
distinguished triangle in $\Stmodg$ having the form
\begin{equation} \label{eq:cann}
\xymatrix{ 
\CE_{\CM}(X) \ar[r]^{\quad \theta_X} & X \ar[r]^{\mu_X \quad} & 
\CF_{\CM}(X) \ar[r] &  \Omega^{-1}(\CE_{\CM}(X))
}
\end{equation}
and having certain universal properties \cite{R}.
Let $\CM^{\oplus}$ denote the closure of $\CM$ in $\Stmodg$ under 
arbitrary direct sums. The map $\theta_X$ is universal for maps 
from objects in $\CM^\oplus$ to $X$, meaning that if $Y$ is in $\CM^{\oplus}$
then any map $Y \to X$ factors through $\theta_X$. The map 
$\mu_X$ is universal for maps from $X$ to $\CM$-local objects. An object
$Y$ is $\CM$-local if $\Homul_{kG}(M, Y) = \{0 \}$ for all $M$ in 
$\CM$. The universal property says that for a module $Y$  
that is $\CM$-local,  any map $X \to Y$ factors through $\mu_X$. 

For a module $X$, the canonical triangle for $X$ is the tensor product of 
$X$ with the canonical triangle for $k.$ The modules  
$\CE_{\CM}(k)$ and $\CF_{\CM}(k)$ are idempotent module in that 
$\CE_{\CM}(k) \otimes \CE_{\CM}(k) \cong \CE_{\CM}(k)$ and 
$\CF_{\CM}(k) \otimes \CF_{\CM}(k) \cong \CF_{\CM}(k)$ in the stable category.
In addition, the two are orthogonal meaning that 
$\CE_{\CM}(k) \otimes \CF_{\CM}(k)$ is projective, {\it i. e.} zero in 
the stable category. 

It is helpful to know the support varieties of these modules. The following
is well known, but we sketch a proof. 

\begin{prop} \label{prop:varidem}
Suppose that $\CV$ is a collection of subvarieties of $\CV_G(k)$ that is closed
under specialization. Let $\CM = \CM_\CV$,
the thick tensor ideal of all finitely generated $kG$-modules $M$ such 
that $\CV_G(M)$ is in $\CV$. Then $\CV_G(\CE_{\CM}(k)) = \CV$ and 
$\CV_G(\CF_{\CM}(k)) = \CV_G(k) \setminus \CV$. 
\end{prop}

\begin{proof}
The fact that $\CE_{\CM}(k) \otimes \CF_{\CM}(k)$ is projective, implies 
that their support varieties are disjoint. The fact that the trivial module
is the third object in a triangle involving the two implies that 
$\CV_G(\CE_{\CM}(k)) \cup \CV_G(\CF_{\CM}(k)) = \CV_G(k)$. 
If $V$ a closed subvariety of  $\CV$, then the universal property says that 
that the identity homomorphism of a finitely generated module $M$ with 
$V_G(M) = V$ factors through $M \otimes \CE_{\CM}(k)$. Thus $V \in 
\CV_G(\CE_{\CM}(k))$. On the other hand, if $V \in \CV_G(k)$ is not in 
$\CV$, then there is a module $M$ in $\Stmodg$ 
whose support variety is the one 
(generic) point $V$. Then $M$ is $\CM$-local, and the universal property 
implies that $M$ is a direct summand of $M \otimes \CF_{\CM}(k)$. So 
$V$ is not in $\CV_G(\CE_{\CM}(k))$.
\end{proof} 

Choose a nonnilpotent element $\zeta \in \HHH^n(G,k)$ for some $n > 0$ and
let $V = V_G(\zeta)$ be the variety of the ideal generated by $\zeta$. 
Note that $\zeta$ is represented by a cocycle $\zeta: k \to \Omega^{-n}(k)$.
For convenience, we denote the shifts of this map $\Omega^t(k) \to 
\Omega^{t-n}(k)$, also by $\zeta$.  Let 
$\CM_V$ be the thick tensor ideal consisting of all finitely generated 
$kG$-modules $M$ with $V_G(M) \subseteq V$. The cohomology of any element
in $\CM$ is annihilated by a power of $\zeta.$ As a consequence, for
$X$ in $\stmodg$, any map $\tau:k \to X$, whose third object in its triangle 
is in $\CM$, has the property that it is a factor of 
$\zeta^t: k \to \Omega^{-tn}(k)$ for some $t$, sufficiently large. 
That is, there is a map $\beta: X \to \Omega^{-tn}(k)$ such that 
$\zeta^t = \beta\tau$.

Thus it can be shown \cite{R} that the module $\CF_{\CM}(k)$ can be taken to 
be the direct limit (to be precise, we take a homotopy colimit) of the system 
\[
\xymatrix{
k \ar[r]^{\zeta} & \Omega^{-n}(k) \ar[r]^{\zeta} & 
\Omega^{-2n}(k) \ar[r]^{\zeta} & \Omega^{-3n}(k) \ar[r]^{\zeta} & \dots
}
\]
Likewise, $\CE_{\CM}(k)$ can be taken to be the homotopy colimit of the 
third objects in the triangles of the maps $\zeta^t: k \to \Omega^{-tn}(k)$. 
It is the third object in the triangle $k \to \CF_{\CM}(k)$. For $X$ 
in $\Stmodg$, the canonical triangle \ref{eq:cann} involving $X$ is the 
tensor product of $X$ with this one. Also $\CF_{\CM}(X)$ is $\CM$-local 
and the universal properties are satisfied. 

From all of this, it is routine to show the following (see \cite{R}).

\begin{prop} \label{prop:trivendo}
Let $\CC$ be the Verdier localization of the category $\stmodg$ at the 
thick tensor ideal $\CM_V$ for $V = V_G(\zeta)$ as above. Then the 
endomorphism ring $\Hom_{\CC}(k,k)$ of the trivial module $k$ is the 
degree zero part of the localized cohomology ring $\hgs[\zeta^{-1}]$. 
\end{prop}

\begin{proof}
Recall that there is a natural identification 
$\Hom(k, \Omega^{-m}(k)) \cong \HHH^m(G,k)$ for any $m \geq 0$.
Choose an element in $\Hom_{\CC}(k,k)$. It must have the form 
$\tau^{-1}\gamma$ where $\gamma: k \to X$ and $\tau: k \to X$ has 
the property that the third object in the triangle of $\tau$ is 
in $\CM_V$. Then as above, for some $t$ there exist $\zeta: X 
\to \Omega^{-tn}(k)$ such that $\zeta^t = \tau\beta$. So we have a 
diagram 
\[
\xymatrix{ 
k \ar[r]^\tau & X \ar[d]^\beta & k \ar[l]_\gamma \\
&\Omega^{-tn}(k)
}
\]
and $\tau^{-1}\gamma =  (\beta\tau)^{-1}(\beta\gamma) = \zeta^{-t}\beta\gamma$
where $\beta\gamma$ is an element of $\HHH^{tn}(G,k)$. 
\end{proof}

We end this section with a straightforward calculation that is needed 
later. 

\begin{prop}  \label{prop:rank2}
Suppose that $G = \langle y,z \rangle$ is an elementary abelian group 
of order $p^2$, and $H = \langle y \rangle$. 
Let $k$ be a field of characteristic $p$. Let $P_*$ be a minimal
$kH$-projective resolution of the trivial module $k$. Let $V$ be the 
variety corresponding to the point defined by the subgroup $\langle z \rangle$,
and let $\CM_V$ be the thick tensor ideal of finitely generated module $M$
such that $V_G(M) = V$. Then $\CE_{\CM}(k) \cong M(P_*, \infty)$,
$\CF_{\CM}(k) \cong N(P_*, \infty)$ and 
the canonical triangle of $k$ as in \ref{eq:cann} is the triangle as given
in Proposition \ref{prop:canon-tri},  
defined by the augmentation map $\varepsilon: P_* \to k$.  
\end{prop}

\begin{proof}
First notice that in the minimal resolution $P_*$, every $P_i \cong kH$, 
a $p$ dimensional module with basis consisting of $1, Y, \dots Y^{p-1}$ 
where $Y = 1+y$. The map 
$\partial: P_{2i+1} \to P_{2i}$ takes $1$ to $Y$ and 
$\partial: P_{2i} \to P_{2i-1}$ takes $1$ to $Y^{p-1}.$  Let $Z = 1+z$.
Hence, $N(P_*,n)$ has a basis $u_0, \dots, u_{2n-1}$
where $u_{2i} = (1, 0, \dots, 0) \in P_{2i}^{p-1}$, 
and $u_{2i+1} = 1 \in P_{2i+1}$
for $i = 0, \dots, n-1$. Thus, we can see from the definition of $N(P_*,n)$
that 
\[
Yu_{2i}= -Zu_{2i+1} \ \text{ for } 0 \leq i \leq n-1  \text{  and  }
Y^{p-1}u_{2i-1} = -Z^{p-1}u_{2i} \ \text{ for } 1 \leq i \leq n-1.
\]  
The map $\varepsilon: k \to N(P_*, n)$ induced 
by the augmentation $P_* \to k$ takes $1 \to Z^{p-1}u_0$. 

An injective $kH$-resolution of $k$ has the form 
\[
\xymatrix{
0 \ar[r] & k \ar[r]^{\epsilon} & R_0 \ar[r] & R_1 \ar[r] & \dots
}
\]
where every $R_i \cong kH$, $\epsilon(1) \in Y^{p-1}R_0$, and the boundary 
maps alternate between multiplication by $Y$ and by $Y^{p-1}$. Similarly, for
$A = kZ/(Z^p)$ an $A$-injective resolution of $k$ has the form 
$0 \to k \to Q_0 \to Q_1 \dots$ where every $Q_i \cong A$ and the maps are
as above with $Z$ substituted for $Y$. Thus, a $kG$-injective resolution of 
the $k$ is the tensor product $Q_* \otimes R_*.$ Now,  $\Omega^{-2n}(k)$
is the quotient 
\[
\Omega^{-2n}(k) = (Q_* \otimes R_*)_{2n-1}/\partial((Q_* \otimes R_*)_{2n-2}).
\]
This module is generated by the classes of the element $v_i = 1 \otimes 1$
in $Q_i \otimes R_{2n-i-1}$ for $i= 0, \dots, 2n-1$. We have that for 
$1 \otimes 1 \in Q_{2i}\otimes R_{2(n-i)-2}$, 
\[
\partial(1 \otimes 1) = Z \otimes 1 \ + \ 1 \otimes Y 
\quad \in (Q_{2i+1}\otimes R_{2(n-i)-2}) \ \oplus \ 
(Q_{2i} \otimes R_{2(n-i)-1}),
\]
while for $1 \otimes 1$ in $Q_{2i-1} \otimes R_{2(n-i)-1}$
\[
\partial(1 \otimes 1) = Z^{p-1} \otimes 1 \ - \ 1 \otimes Y^{p-1}
\quad \in (Q_{2i}\otimes R_{2(n-i)-1}) \ 
\oplus \ (Q_{2i-1} \otimes R_{2(n-i)}).
\]
Thus we have relations $Zv_{2i+1} = - Yv_{2i}$ and 
$Z^{p-1}v_{2i} = Y^{p-1}v_{2i-1}$.  
These are (except for signs) the same relations
as for $N(P_*,n)$, and hence $\Omega^{-2n}(k) \cong N(P_*,n)$. 

Next notice that, with the above identification, the map $k \to N(P_*,n)$
takes $1 \in k$ to the class of $Z^{p-1}v_0$ which is contained in 
$Q_0 \otimes R_{2n-1}$.  In the case that $n=1$, this is the 
cohomology class of the inflation to $G$ of the polynomial generator $\zeta$
in degree 2 of $\hhs$. For $n>1$ it represents the class of $\zeta^n$. 
Thus we have a sequence of maps 
\[
\xymatrix{
k \ar[r]^\zeta & \Omega^{-2}(k) \ar[r]^\zeta & \Omega^{-4}(k) \ar[r]^\zeta 
& \Omega^{-6}(k) \ar[r] & \dots 
}
\]
By the construction of \cite{R} that is summarized at the beginning 
of this section, we have that the canonical triangle associated to 
$\CM_V$ is the triangle of the map $k \to N(P_*,\infty)$. 
This proves the proposition. 
\end{proof}


\section{Canonical triangles} \label{sec:cann}
Assume as before 
that $kG = kH \otimes kC$  where $kH$ and $kC$ are Hopf subalgebras and 
$kC \cong k[Z]/(Z^p)$.
Let $\CW$ be the point
in $\CV_G(k)$ which is the  image of the restriction map
$\res_{G, C}^*: \CV_C(k) \to \CV_G(k)$. That is, $\CW$ is the equivalence
class of the inclusion $k[Z]/(Z^p) \to kG$ viewed as a $\pi$-point. 
Let $\CM = \CM_\CW$ be the thick tensor ideal of $\stmodg$ consisting of all 
modules whose variety is in $\CW$. Thus a module in $\CM$ is
a projective module on restriction to a $kH$-module and is not projective
when restricted to $kC$.

Let $P_*$ be a $kH$-projective resolution of the trivial $kH$-module $k$.
As in the last section, let $\CE = \Gamma(k) = M(P_*, \infty)$
and $\CF = N(P_*,\infty)$. Let  
\[
\xymatrix{
\CS: \qquad  & \CE \ar[r]^{\theta} & k \ar[r]^{\mu} & \CF \ar[r] &
\Omega^{-1}(\CE).
}
\]
be the triangle of the exact sequence in Proposition \ref{prop:canon-tri}.

For any $X \in \Stmodg$, tensoring with $\CS$ we obtain  a triangle
\[
\xymatrix{
X \otimes \CS: & \CE(X) \ar[r] & X \ar[r] \ar[r] &
\CF(X) \ar[r] & \Omega^{-1}(\CE(X))
}
\]
Our object is to show that the triangles $X \otimes \CS$ satisfy certain
universal properties. This allows us to calculate the endomorphisms in the
Verdier localization at the subcategory $\CC$. The first step is to
establish the support varieties of the modules $\CE$ and $\CF$.

\begin{prop} \label{prop:varidem2}
The variety of $\CE$ is $\CV_G(\CE) = \CW$, 
the set consisting of the single equivalence class or 
$\pi$-points as above.  The variety of $\CF$ is
$\CV_G(k) \setminus \CW$.
\end{prop}

\begin{proof}
Suppose that $K$ is an extension of $k$ and let $\alpha_K:K[t]/(t^p) \to 
KG$ be a $\pi$-point.  Let $\beta: k[t]/(t^p) \to KG$ be the $\pi$ point 
given by $\beta(t) = Z$. Our objective is to show that if $\alpha_K$ is not
equivalent to $\beta$ then the class of $\alpha_K$ is not in the variety 
of $\CE$. We know that the class of $\beta$ is in $\CV_G(\CE)$. So assume
that $\alpha_K$ is not equivalent to $\beta$. Recall that $\CV_G(k) 
\cong \Proj \hgs$. The equivalence send the class of $\alpha_K$ to the 
variety of the kernel of the induced map from the 
cohomology of $G$ to that of the 
$A_K = K[t]/(t^p)$. Recall that 
$\HHH^*(A,K)/\Rad(\HHH^*(A,K)) \cong K[T]$, a polynomial 
ring, and $\HHH^*(G, K) \cong \HHH(H,K) \otimes \HHH^*(C,K)$. Moreover, 
since $C$ is a cyclic group of order $p$, 
$\HHH^*(C,K)/\Rad(\HHH^*(C,K)) \cong K[\zeta]$ is 
a polynomial ring in the degree two element $\zeta$. In particular, two 
$\pi$-points are in the same class if varieties of the corresponding 
kernels are the same. 

For the $\pi$-point $\alpha_K$, let $\varphi_\alpha: \HHH^*(G,K) \to 
\HHH^*(A,K) \cong K[T]$ be the map induced by the restriction. 
Let $\varphi:\HHH^*(H,K) \to K[T]$ be the restriction of 
$\varphi_\alpha$ to $kH$ and then inflated to $kG$. That is, we 
restrict $\varphi_\alpha$ to the subring $\HHH(H,K) \otimes 1$ of
$\HHH^*(G,K)$ and compose this with the map $\HHH^*(G,K) \to 
\HHH^*(H,K)$. Notice that if the image of $\varphi$, which is a 
subring of $K[T]$, is only the field $K$, then $\alpha_K$ is equivalent
to $\beta$. That is, in such a case, if $\eta \otimes \zeta^n \in 
\HHH(H,K) \otimes \HHH^*(C,K)$ and the degree of $\eta$ is greater than 
zero, then $\varphi(\eta \otimes \zeta^n) = \varphi(\eta \otimes 1)
\varphi(1 \otimes \zeta^n) = 0$. Thus $\alpha_K$ and $\beta$ correspond 
to the same element of $\Proj \hgs$. Hence, we may assume that the 
kernel of $\varphi$ is a nonzero prime ideal in $\HHH^*(H,K)$,
which is not the ideal of all positive elements. 

Let $\gamma_K: k[t]/(t^p) \to KH \otimes 1 \subseteq kG$ 
be a $\pi$ point corresponding to 
$\varphi$. Let $KE = K[u,v]/(u^p, v^p)$ and let $\mu:KE
\to KG$ be defined by $\mu(t) = \gamma_K(u)$ and 
$\mu(u) = \beta_K(v) = 1 \otimes Z$. Note that $\mu(u)$ and $\mu(v)$ 
commute so that the given conditions on $\mu$ define a homomorphism.
Moreover, $\mu$ is a flat embedding since $\gamma_K$ is a flat 
embedding and $\beta_K(K[t]/(t^p)) = 1 \otimes KC$. 
Note especially that 
the kernel of $\varphi_\alpha$ is contained in the kernel of 
$\gamma_K^*: \HHH^*(G,K) \to \HHH^*(E,K)$. Thus there is some 
$\pi$-point $\hat{\alpha}_K: K[t]/(t^p) \to KE$ such that the 
composition $\gamma_K\hat{\alpha}_K$ is equivalent to $\alpha_K$. 

We consider the restriction $M = \mu^*(\CE)$ to a $KE$-module. 
The projective resolution $K \otimes P_*$ restricted to $KE$ is 
$KE$-projective resolution of $K$. Hence, 
we have that $M \cong M(K \otimes P_*, \infty)$.
Hence, from Lemma \ref{lem:resol1}, we know the variety of $M$ 
and know also that the third object in the triangle of $\theta:M \to K$
is $N = N(K \otimes P_*, \infty)$. 
The variety of consists only of the class of the $\pi$-point 
$\hat{\beta}: K[t]/(t^p) \to KE$ given by $t \mapsto v$. It 
follows that $\hat{\alpha}_K$ is not in the variety of $N$ and that
$\alpha_K$ is not in the variety of $\CE$. Since the restriction to 
$kE$ takes triangles to triangles, we have also that $\alpha_K$ is
in the vareity of $\CF$. This proves the theorem. 
\end{proof}

We can now establish the universal properties.

\begin{thm} \label{thm:uni-resol}
Assume the hypotheses and notation of this section. For any $kG$-module $X$,
the triangle $X \otimes \CS$ is the canonical triangle as in \ref{eq:cann}
for $X$ relative to $\CM$. In particular, $1_X\otimes \theta$ is universal
with respect to maps from objects in $\CM^{\oplus}$ to $X$ 
and $1_X \otimes \mu$
is universal with respect to maps from X to $\CM$-local objects. 
\end{thm}

\begin{proof}
First note that $\CE \otimes \CF$ is a projective module, zero in the
stable category, because the intersection of the varieties of the two
modules is empty. Thus tensoring $\CS$ with $\CE$ or $\CF$, we see that
$\CE$ and $\CF$ are idempotent modules.

Suppose next that $M$ is in
$\CM$. Then
\[
\Homul_{kG}(M, X \otimes \CF) =
\Homul_{kG}(k, M^*\otimes X \otimes \CF) - \{0\}.
\]
since $M$ is finitely generated and the dual $M^*$ of $M$ and $\CF$ have
disjoint varieties. Thus, $X \otimes \CF$ is $\CM$-local, and
by Lemma 5.2 of \cite{R}, $X \otimes \CF$ is $\CM^\oplus$-local.
We claim that the map from $X$ to $X \otimes \CF$ is universal
map from $X$ to $\CM$-local objects. The reason is that if $N$ is
$\CM$-local then $\Homul_(kG)(X\otimes \CE, N) = \{0\}$ and so
$\Homul_{kG}(X \otimes \CF, N) \cong \Homul(X, N)$, the isomorphism
being induced by the map $X \to X \otimes \CF$. So any map from
from $X$ to $N$ factors through $1_X \otimes \mu$.

Likewise, we can see that the map $X \otimes \CE \to X$ is universal with
respect to map from an object $Y$ in $\CM^\oplus$ to $X$. That is, if $M$ is
in $\CM^\oplus$, then because $\CF$ is $\CM^\oplus$-local, the
canonical triangle for $X$ yields an isomorphism $\Homul_{kG}(Y, X \otimes \CE)
\cong \Homul_{kG}(Y,X)$. Thus any map from $Y$ to $X$ factors through
$1_X \otimes \theta$. This proves the theorem.
\end{proof}


\section{The endomorphism ring of the trivial module} \label{sec:endo}
In this section we assume the hypotheses and notation of 
Section \ref{sec:cann}, just previous. It is well known that the 
endomorphism ring of the trivial module in the localized category 
is associated to the structure of the module $\CF$. 
Here is where we use our development of the structure of $\CF$. 

Let $G = H \times C$ as before. 
Let $\CM$ denote the thick tensor ideal in 
$\stmodg$ consisting of all $kG$-modules whose variety is the point
in $V_G(k)$ which is the  image of the restriction map
$\res_{G, C}^*: V_C(k) \to V_G(k)$. Let $\CC$ be the Verdier localization
of $\stmodg$ at $\CM$. We are interested in the ring $\Homul_{kG}(k,k) =
\Hom_{\CC}(k,k)$. Given a morphism $k \rightarrow^\sigma M \leftarrow^\tau
k$ with the third object of the triangle of $\sigma$ in $\CM$, we get a 
diagram 
\[
\xymatrix{
& U \ar[d]^\nu \\
\CF \ar[r] & k \ar[r]^\mu \ar[d]^\sigma & \CF \\
& M
}
\]
with $U$ in $\CM$. Because $\CF$ is $\CM$-local the composition 
$\mu\nu$ is the zero map. Thus, there is a map $\varphi: M \to \CF$ 
such that $\varphi\sigma = \mu$. The morphism $\sigma^{-1}\tau$
is equal to one of the form $\mu^{-1}\alpha$ for $\alpha = \varphi\tau$.
Thus every endomorphism of $k$ is factors through $\CF$. 

Note that we need not worry that $\CF$ is infinitely generated. In the 
above, we could have replaced $\CF$ with the submodule whose restriction
to $kH$ is $N(P_*,n)_{\downarrow H} \cong k \oplus P_0^{p-1} \oplus 
\dots \oplus P_{2n-1}$ for $n$ sufficiently 
large. By the same argument as above, this is the third object in the 
triangle of the map $M(P_*,n) \to k$. The point of using $\CF$ is that 
we have a context to compute compositions of morphisms. 

So suppose we have two endomorphisms of $k$ in $\CC$, $\mu^{-1}\alpha$
and $\mu^{-1}\beta$. For the purposes of computing the product we may 
assume that there exist $m$ and $n$ such that $\alpha(k) \subseteq P_m$ 
and $\beta(k) \subseteq P_n$. That is, otherwise $\alpha$ and $\beta$ 
can be written as sums of such elements. The product is the map 
$\mu^{-1}\alpha^\prime \beta$ in the diagram
\[
\xymatrix{
k \ar[dr]^\mu && k \ar[dl]^\alpha \ar[dr]^\mu && k \ar[dl]^\beta \\
& \CF \ar[dr]^{Id} && \CF \ar[dl]^{\alpha^\prime} \\
&& \CF
}
\]
The construction of $\alpha^\prime$ is through a chain map on complexes
\[
\xymatrix{
\dots \ar[r]  & P_2 \ar[r] \ar[d]  & P_1 \ar[r] \ar[d] &
P_0  \ar[r]^{\varepsilon} \ar[d] & k \ar[r] \ar[d]^\alpha & 0 \\
\dots \ar[r] & P_{m+2} \ar[r]^\partial & P_{m+1} \ar[r]^\partial &
P_m   \ar[r]^{\partial} & P_{m-1} \ar[r] & \dots \ar[r] & P_0 \ar[r] &
k \ar[r] & 0
}
\]
obtained by lifting the map $\alpha$. Then the chain map induces the 
map $\alpha^\prime: \CF \to \CF$ as in Lemma \ref{lem:resol1}. It is 
clear that the diagram commutes, {\it i. e.} $\alpha^\prime \mu = \alpha$. 

All of this sets up the proof of our main theorem. 

\begin{thm} \label{thm:trivendo}
In the localized category $\CC$ the endomorphism ring of $k$ is isomorphic
to the negative Tate cohomology ring of $H$:
\[
\Homul_{kG}(k,k) \cong \widehat{\HHH}^{\leq 0}(H, k).
\]
\end{thm}

\begin{proof}
Because the projective resolution $P_*$ is minimal, we know that 
$\Hom_{kG}(k, P_n) \cong \widehat{\HHH}^{-n-1}(H,k).$ The Tate cohomology 
group $\widehat{\HHH}^{-n-1}(H,k)$ is also isomorphic to homotopy classes
of chain maps of degree $-n-1$ from the augmented complex  
$(P_*, \varepsilon)$ to itself  
as in the diagram. Indeed that diagram defines the 
correspondence. The product of two elements is defined (see \cite{BC2}) 
as the composition of the chain maps. Thus the element represented 
by $\alpha^\prime \beta$ is the product of the elements represented 
by $\alpha$ and $\beta$ in the Tate cohomology.
\end{proof}

There are a few cases in which we can be very specific about the structure
of the endomorphism ring of the trivial module in these localized
categories.  We end the section with two examples and a remark on the 
role of Hopf structures.
The reader should recall that a group
has periodic cohomology (in characteristic p) if and only if its
Sylow $p$-subgroup is either cyclic or quaternion (with $p=2$).

\begin{prop} \label{prop:periodic}
Suppose that $H$ has periodic cohomology meaning that the
cohomology ring $\HHH^*(H,k)$ has Krull dimension one. Then
the endomorphism ring of $k$ is the degree zero part of a localization
of the cohomology ring of $G$. Specifically,
\[
\Homul_{kG}(k,k) \cong \Hom_{\CC}(k,k) \cong
\sum_{n \geq 0} \HHH^{nm}(G,k)\zeta^{-n}
\]
where $\zeta$ is an element is a regular element of degree $m$ in
$\hhs$.
\end{prop}

\begin{proof}
Because $G = H \times C_p$ is a direct product,
$\hgs \cong \hhs \otimes \HHH^*(C_p,k)$. If $\zeta \in \HHH^m(H,k)$ is
a regular element, the variety $V$ is $V = V_G(\zeta)$, where we identify
$\zeta$ with $\zeta \otimes 1$. So the proposition follows from the
discussion in Section \ref{sec:idmod}. We note that this result is
compatible with Theorem \ref{thm:trivendo} since $\sum_{n \leq 0}
\widehat{\HHH}^{-n}(H,k) \cong
\sum_{n \leq 0}\widehat{\HHH}^{-n}(H,k)\gamma^n$ where $\gamma$ is a
degree two generator for $\HHH^*(C_p,k)$.
\end{proof}

On the other hand, suppose that $H$ has $2$-rank at least $2.$ Then we
get a very different result.

\begin{prop} \label{prop:eleab}
Suppose that $H$ is an elementary abelian $p$-group of order at least $p^2$.
Then $\Hom_{\CC}(k,k)$ is a local $k$-algebra whose radical is infinitely
generated and has square zero.
\end{prop}

\begin{proof}
By \cite{BC2}, the  product of any two elements in negative cohomology
for $H$ zero. So by Theorem \ref{thm:trivendo}, the radical of $\Hom_{\CC}(k,k)$
which consist of all elements in negative degrees, has square zero.
At the same time, $\widehat{\HHH}^{-n}(H,k)$ is Tate dual to
$\HHH^{n-1}(H.k)$ and hence its dimension grows as $n$ becomes large.
\end{proof}

\begin{rem} \label{rem:coalg}
The assumption throughout this article has been that $G = 
H \times C$ where $H$ and $C$ are subschemes, {\it i. e.} that $kH$ and 
$kC$ are cocommutative sub-Hopf algebras of $kG$. 
Even though the coalgebra structure played a role in
the proofs, it has no role in the statement of the main theorem. Neither
the endomorphism ring $\End_{\CC}(k)$ nor the structure of the negative
cohomology ring depend on even the existence of a coalgebra structure.
Hence, if we have two algebras $kH$ and $kC$ on which we can create 
coalgebra maps of the right sort, then we come to the same result. 
For example, suppose that $G$ is an elementary abelian $p$-group and
that $\alpha:k[t]/(t^p)$ is a $\pi$-point defined over $k$. Let $kC$ be
the image of $\alpha$. Then $kC$ has a complement $kH$ such that 
$kH$ is the group algebra of an elementary abelian $p$-group of 
one smaller rank than $G$ and 
$kG \cong kH \otimes kC$. Let $\CM$ denote the thick tensor ideal of 
modules whose support varieties are either empty or consist only of
the class of $\alpha$,  and let $\CC$ be the localization of 
$\stmodg$ at $\CM$. Then the main theorem applies. 
\end{rem}


\end{document}